\documentclass[12pt, leqno]{amsart}
\usepackage{amssymb,url}
\usepackage[mathscr]{eucal}
\usepackage{mathrsfs}
\usepackage{mystyle}
\usepackage{enumitem}
\usepackage{multicol}
\begin{document}
\title[Product formula for limits of normalized characters of KR modules]{Product formula for the limits of normalized characters of Kirillov-Reshetikhin modules}
\author{Chul-hee Lee}
\address{School of Mathematics, Korea Institute for Advanced Study, Seoul 130-722, Korea}
\email{chlee@kias.re.kr}
\date{\today}
\begin{abstract}
The normalized characters of Kirillov-Reshetikhin modules over a quantum affine algebra have a limit as a formal power series. Mukhin and Young found a conjectural product formula for this limit, which resembles the Weyl denominator formula. We prove this formula except for some cases in type $E_8$ by employing an algebraic relation among these limits, which is a variant of $Q\widetilde{Q}$-relations.
\end{abstract}
\maketitle
\section{Introduction}
Let $\g$ be a complex simple Lie algebra. Kirillov-Reshetikhin (KR) modules form an important family of finite-dimensional irreducible representations of the quantum affine algebra $\uqghat$. Many important objects and results from solvable lattice models in mathematical physics can find a rigorous mathematical foundation in the representation theory of $\uqghat$, in which KR modules play a key role. A good example is the recent progress \cite{MR3397389} on understanding the spectra of Baxter's $Q$-operators \cite{MR0290733}.
In \cite{MR2982441} Hernandez and Jimbo introduced a certain category $\mathcal{O}$ of representations of a Borel subalgebra of $\uqghat$. One of the most important objects in $\mathcal{O}$ is called a prefundamental representation, which is infinite-dimensional and is obtained as the limit of a sequence of KR modules via an asymptotic construction. It turns out that Baxter's $Q$-operators acquire a solid representation theoretic background in terms of prefundamental representations. See \cite{MR3576119, 2016arXiv160605301F} also for more recent developments.

The construction of prefundamental representations in \cite{MR2982441} was partly motivated by the fact that the normalized $q$-characters of KR modules have limits as formal power series. Nakajima and Hernandez have proved this convergence property for simple-laced cases \cite{MR1993360} and in general \cite{MR2254805}, respectively. This limit can now be understood, for example, as the normalized $q$-character of a prefundamental representation \cite[Section 6.1]{MR2982441}. For ordinary characters, this implies the existence of the following limit
\begin{equation}\label{eqn:norchar}
\nc{a} : = \lim_{m\to \infty} e^{-m\omega_a}\Qam
\end{equation}
as a formal power series in $\Z[[e^{-\alpha_j}]]_{j\in I}$, where
$\Qam\in \Z[P]$ denotes the character of the KR module associated with $a\in I$ and $m\in \Z_{\geq 0}$. Here $I$ is the set of nodes of the Dynkin diagram of $\g$. The study of this limit goes back to \cite{MR1745263} and \cite{MR1903843}, in which the authors prove that its existence plays a critical role in establishing the Kirillov-Reshetikhin conjecture \cite{Kirillov1990}.

In \cite{MR3217701} Mukhin and Young conjectured that
\begin{equation}\label{eqn:main}
\nc{a} = \ncMY{a},
\end{equation}
where
\begin{equation}\label{eqn:MY}
\ncMY{a} := \frac{1}{\prod_{\alpha\in \posR}(1-e^{-\alpha})^{[\alpha]_a}}.
\end{equation}
Here, $[\alpha]_a\in \Z_{\geq 0}$ denotes the coefficient in the expansion $\alpha =\sum_{a\in I}[\alpha]_a\alpha_a$. We note that their conjecture is for more general minimal affinizations, not just for KR modules of $\uqghat$. A different version of an explicit formula, with the flavor of fermionic formula, is also proved in \cite[Theorem 6.4]{MR2982441}, but (\ref{eqn:MY}) looks much more concrete and compact; see \cite[Remark 4.19]{MR3500832} also for a geometric $q$-character formula for prefundamental representations. The main goal of this paper is to prove the following :
\begin{thm}[{\cite[Conjecture 6.3]{MR3217701}}]\label{thm:main}
Let $\g$ be a simple Lie algebra and $a\in I$ be a node in its Dynkin diagram. We assume that $a\notin \{4,8\}$ when $\g$ is of type $E_8$ (see Figure \ref{fig:E8}). Then (\ref{eqn:main}) holds.
\end{thm}
Previously, (\ref{eqn:main}) was proved for type $A_r$, $B_r$, $C_r$ in \cite{MR3120578} and for type $G_2$ in \cite{MR3520056}. In view of (\ref{eqn:MY}), it is not surprising to see that the combinatorics of roots appear in some important steps in these works. In fact, once there is a polyhedral formula available for a given node, proving (\ref{eqn:main}) is equivalent to checking some combinatorial identities similar to the Weyl denominator formula; see Subsection \ref{ss:polyform} for the meaning of polyhedral formula. This verification is still not an entirely automatic procedure in general; see Propositions in Section \ref{sec:MYandpoly} to get some flavor of the problem. Although this problem in full generality seems to be an interesting subject in its own right, it seems to require a separate combinatorial consideration for each type at this point, making it hard to obtain a uniform proof of (\ref{eqn:main}).

In this paper, we take a more uniform approach to (\ref{eqn:main}), minimizing such combinatorial consideration. The key ingredient is the following relation for a family $X=\left(\X{a}\right)_{a\in I}$ :
\begin{equation}\label{eqn:norsys}
\Xrel{a} = \prod_{b:C_{ab}<0}(\X{b})^{-C_{ab}},\quad a\in I,
\end{equation}
to which $\left(\nc{a}\right)_{a\in I}$ is supposed to be a solution. Intuitively, (\ref{eqn:norsys}) is obtained by extracting the dominant exponential terms from both sides of the $Q$-system (\ref{eq:Qsys}). It requires some work to make this intuition rigorous (\ref{eqn:norsys}) looks very similar to the $Q\widetilde{Q}$-system \cite[Theorem 3.2]{2016arXiv160605301F}, but the fact that $\left(\nc{a}\right)_{a\in I}$ is a solution of (\ref{eqn:norsys}) does not seem to be an immediate consequence of the results in there; see Remark \ref{rmk:QQcomparison} for more on this. This relation involving only product looks quite compatible with the product form (\ref{eqn:MY}) of $\ncMY{a}$. It is indeed easy to show that $\left(\ncMY{a}\right)_{a\in I}$ satisfies (\ref{eqn:norsys}); see Proposition \ref{prop:relforMY}. We can use this to prove (\ref{eqn:main}) even in the absence of the corresponding polyhedral formula. For example, if $a\in I$ has a unique node $b\in I$ with $C_{ab}<0$, and moreover $C_{ab}=-1$, then (\ref{eqn:norsys}) takes the form
$$
\Xrel{a} = \X{b}.
$$
Then we can easily conclude that $\nc{b} = \ncMY{b}$ once we already know $\nc{a} = \ncMY{a}$. In this way we get a relatively simple proof of (\ref{eqn:main}), which works quite well even when $\g$ is of exceptional type. In summary, the steps to prove (\ref{eqn:main}) consists of establishing it for the simplest nodes $a\in I$, and using (\ref{eqn:norsys}) inductively to treat more complicated cases. It might be interesting to establish a more direct link between (\ref{eqn:norsys}) and \cite[Theorem 3.2]{2016arXiv160605301F}, and find analogous relations for more general minimal affinizations, in a form useful to the conjecture of Mukhin and Young.

This paper is organized as follows. In Section \ref{sec:framework}, we set up our notation and explain the necessary background. In particular, we introduce some analogues of $\nc{a}$, which could be defined as a consequence of certain properties of linear recurrence relations among the characters of KR modules; see Subsection \ref{ss:aNC}. In Section \ref{sec:algrelnc}, we study (\ref{eqn:norsys}). In Section \ref{sec:MYandpoly}, we prove (\ref{eqn:main}) for the simplest nodes $a\in I$. In Section \ref{sec:mainproof}, we finish the proof of Theorem \ref{thm:main} by combining the results from Section \ref{sec:algrelnc} and Section \ref{sec:MYandpoly}.

\section{Limits of normalized characters and their analogues}\label{sec:framework}
The main goal of this Section is to define the limits of normalized characters and their analogues for KR modules; see Definition \ref{def:anc}. Since our treatment of these objects occasionally requires the use of limits, we also consider their basic analytic properties. 
\subsection*{notation}
Throughout the paper, we will use the following notation.
\begin{itemize}
\item $\g$ : simple Lie algebra over $\C$ of rank $r$
\item $\h$ : Cartan subalgebra of $\g$
\item $I=\{1,\dots, r\}$ : index set for the Dynkin diagram of $\g$ (we use the same convention as \cite{MR1745263} except for $E_8$; for $E_8$ see Figure \ref{fig:E8})
\item $\alpha_a,\, a\in I$ : simple root 
\item $h_a,\, a\in I$ : simple coroot 
\item $\omega_a,\, a\in I$ : fundamental weight
\item $C = (C_{ab})_{a,b\in I}$ : Cartan matrix with $C_{ab}=\alpha_b(h_a)$
\item $\rho = \sum_{a\in I}\omega_a$ : Weyl vector
\item $P$ : weight lattice, 
\item $P^{+}\subseteq P$ : set of dominant integral weights  
\item $L(\la),\la \in P^{+}$ : finite-dimensional simple module of highest weight $\la$ of $\g$ (and $\uqg$)
\item $Q$ : root lattice
\item $\theta\in Q$ : highest root
\item $\h_{\R}^{*}:=\oplus_{a\in I}\R\omega_a$
\item $(\cdot,\cdot):\h_{\R}^{*}\times \h_{\R}^{*} \to \R$ : $\R$-bilinear form induced from the Killing form with $(\theta,\theta)=2$
\item $\Z[P]$ : integral group ring of $P$ (which is the same as the ring $\Z[e^{\pm \omega_j}]_{j\in I}$ of Laurent polynomials in $e^{\omega_j}$)
\item $\K:=\C(e^{\omega_j})_{j\in I}$ : field of rational functions in $e^{\omega_j}$ with coefficients in $\C$
\item $t_a : =(\theta,\theta)/(\alpha_a,\alpha_a)\in \{1,2,3\}$
\item $[\alpha]_a\in \Z\,(\alpha \in Q, \,a\in I)$ : coefficients in the expansion $\alpha = \sum_{a\in I}[\alpha]_a\alpha_a$
\item $s_a$ : simple reflection acting on $\h_{\R}^{*}$ by $s_a(\la) = \la - \la(h_a)\alpha_a$
\item $W$ : Weyl group generated by $s_a$
\item $W_{\la},\, \la\in P$ : isotropy subgroup of $W$ fixing $\la$
\item $W_{J},\, J\subseteq I$ : standard parabolic subgroup of $W$ generated by $\{s_a : a\in J\}$
\item $L(\la),\, \la \in P^{+}$ : irreducible highest weight representation of $\g$
\item $\chi(V)\in \Z[P]$ : character of a finite-dimensional $\g$-module (or $\uqg$-module) $V=\oplus V_{\la}$ with weight spaces $V_{\la}$, i.e. $\chi(V)=\sum_{\la\in P}(\dim V_{\la})e^{\la}$
\item $\posR$ : set of positive roots
\item $\posR(J) = \{\alpha\in \posR : [\alpha]_a = 0\, \forall a\in I\backslash J\},\, J\subseteq I$, i.e. set of positive roots which can be written as a linear combination of $\{\alpha_j:j\in J\}$
\item $\lambda\geq \mu,\, \lambda,\mu\in \h_{\R}^{*}$ if $\lambda-\mu\in \oplus_{a\in I}\R_{\geq 0}\alpha_a$ 
\item $\lambda\succeq \mu,\, \lambda,\mu\in \h_{\R}^{*}$ if $\lambda-\mu\in \oplus_{a\in I}\Z_{\geq 0}\alpha_a$
\end{itemize}
\begin{figure}
\begin{center}
\begin{tikzpicture}
\begin{scope}[start chain]
\foreach \dyni in {1,...,7} {
\dnode{\dyni}
}
\end{scope}
\begin{scope}[start chain=br going above]
\chainin (chain-3);
\dnodebr{8}
\end{scope}
\end{tikzpicture}
\caption{Dynkin diagram of type $E_8$} \label{fig:E8}
\end{center}
\end{figure}
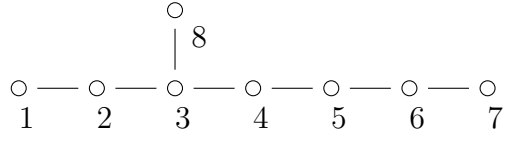

\subsection{Normalized characters of Kirillov-Reshetikhin modules}
Here we mainly collect some results from \cite{MR2254805} that we will use later on. Let $q\in \C^{\times}$ be a complex number which is not a root of unity. For each $(a,m,u) \in I\times \Z_{\geq 0}\times \C^{\times}$, there exists a finite-dimensional irreducible $\uqghat$-module $W^{(a)}_{m}(u)$ called the Kirillov-Reshetikhin module; see, for example, \cite{MR1745263,MR2254805} for a more detailed discussion. By restriction, we get a finite-dimensional $\uqg$-module $\res \Wam(u)$, which we simply denote by $\res \Wam$ as its isomorphism class is independent of $u$. Let $Q_m^{(a)}: =\chi(\res  W_m^{(a)})\in \Z[P]$ for $(a,m)\in I\times \Z_{\geq 0}$ and $\ncQ{a}{m}:=e^{-m\omega_a}\Qam$, which becomes a polynomial in $\Z[e^{-\alpha_j}]_{j\in I}$; see \cite{MR1745260, MR1810773}.

\begin{thm}[{\cite[Theorem 1.1]{MR1993360}} and {\cite[Theorem 3.4]{MR2254805}}]\label{nhmain}
Let $a\in I, m \geq 1$. Then
\begin{equation}\label{eq:Qsys}
(\Qam)^2 - Q^{(a)}_{m+1}Q^{(a)}_{m-1} = \prod_{b : C_{ab}< 0} \prod_{k=0}^{-C_{a b}-1}
Q^{(b)}_{\bigl \lfloor\frac{C_{b a}m - k}{C_{a b}}\bigr \rfloor},
\end{equation}
and as $m\to \infty$, $\ncQ{a}{m}$ converges to $\nc{a}$ as a formal power series in $\Z[[e^{-\alpha_j}]]_{j\in I}$.

\end{thm}
We call (\ref{eq:Qsys}) the $Q$-system. This is a weaker version of what they proved. Nakajima and Hernandez actually proved that the $q$-characters of KR modules satisfy the $T$-system whose restriction becomes the $Q$-system (\ref{eq:Qsys}), and the limit of normalized $q$-characters exists. 

The following gives a more detailed information on how $\ncQ{a}{m}$ changes as polynomials as $m\to \infty$ :
\begin{prop}[{\cite[Lemma 5.8]{MR2254805}}]\label{prop:Epoly}
For each $(a,m)\in I\times \Z_{\geq 0}$, there exists a polynomial $\Eam\in \Z[e^{-\alpha_j}]_{j\in I}$ such that $\ncQ{a}{m+1}-\ncQ{a}{m}=e^{-(m+1)\alpha_{a}}\Eam$ and moreover, $\Eam$ has positive coefficients.
\end{prop}
Let us consider $e^{\la},\, \la \in \h_{\R}^{*}$ as a real-valued function on $\h_{\R}^{*}$ defined by $\mu\mapsto e^{(\la,\mu)}$. Hence, we get $\ncQ{a}{m}(\la)\in \R$ for $\la\in \h_{\R}^{*}$. The argument in the proof of \cite[Proposition 5.9]{MR2254805} shows the following :
\begin{prop}\label{prop:ncfunc}
Let $\dom:=\{\la \in \h_{\R}^{*}: (\alpha_a,\la)>0\, \forall a\in I\}$ be the fundamental open chamber. There exists a non-empty open subset $\dom_0$ of $\dom$ (independent of $a\in I$) such that
for $\la \in \dom_0$, $\ncQ{a}{m}(\la)$ converges absolutely as $m\to \infty$.
\end{prop}
Hence, we can regard $\nc{a}$ as a real-valued functions on $\dom_0$, defined as the limit of $\ncQ{a}{m}$, not just as a formal power series.
\subsection{Analogues of $\nc{a}$}\label{ss:aNC}
In \cite{Lee2016}, we studied a linear recurrence relation with constant coefficients that the sequence $(\Qam)_{m=0}^{\infty}$ satisfies. We can summarize its main properties in terms of its generating function $\mathcal{Q}^{(a)}(t) : = \sum_{m=0}^{\infty}\Qam t^m$ as follows :
\begin{theorem}[{\cite[Theorem 1.1]{Lee2016}}]\label{thm:linQ}
Let $\g$ be a simple Lie algebra. Assume that $a\in I$ belongs to one of the following cases :
\begin{itemize}
\item $a\in I$ is arbitrary when $\g$ is of classical type or type $E_6$, $F_4$ or $G_2$, 
\item $a\in \{1,2,3,5,6,7\}$ when $\g$ is of type $E_7$,
\item $a\in \{1,2,6,7\}$ when $\g$ is of type $E_8$.
\end{itemize}
There exist $W$-invariant finite subsets $\La_a$ and $\La'_a$ of $P$ with the following properties :
\begin{enumerate}[label=(\roman*), ref=(\roman*)]
\item If we set $D^{(a)}(t) :  = \prod_{\la\in \La_a}(1-e^{\la}t)\prod_{\la\in \La_a'}(1-e^{\la}t^{t_a})$, then 
\begin{equation}
N^{(a)}(t): = \mathcal{Q}^{(a)}(t)D^{(a)}(t)
\end{equation}
is a polynomial in $t$ with coefficients in $\Z[P]$ with $\deg N^{(a)}<\deg D^{(a)}$.
\item \label{thm:lineq0} $t_a\La_{a}\cap \La'_{a}=\emptyset$, where $t_a\La_{a}=\{t_a \la\mid \la \in \La_a\}$.
\item $\omega_a\in \La_a$.
\item \label{thm:lineq1} For each $\la\in \La_a$, $\omega_a\geq \la$ .
\item \label{thm:lineq2} For each $\la\in \La_a'$, $t_a\omega_a\geq \la$.
\end{enumerate}
\end{theorem}
From now on we fix $\La_a$ and $\La'_a$ as in the Appendix  of \cite{Lee2016}, which is expected to be the minimal set satisfying the conditions of the above. When $\g$ is simply-laced, $\La_a$ is simply the set of weights of the fundamental representation $L(\omega_a)$. 
Theorem \ref{thm:linQ} was originally motivated to understand a certain periodicity phenomenon related to KR modules \cite{MR3628223}. In particular, \ref{thm:lineq0} shows that $D^{(a)}(t)$ has only simple roots, which is the underlying reason for that periodicity. This fact also turns out to be important in this paper as we discuss below.

Assume that a rational function $\mathcal{R}(t)\in \K(t)$ is of the form
\begin{equation}\label{eqn:Rgen}
\mathcal{R}(t) = \frac{N(t)}{D(t)}, \qquad N,D \in \K[t], \qquad \deg N<\deg D, \qquad D(t) = \prod_{(\la,l)\in \La}(1-e^{\la}t^{l})
\end{equation}
for a finite subset $\La$ of $P\times \Z_{>0}$ such that $D(t)$ has only simple roots. Let $l_0\in \Z_{>0}$ be the least common multiple of $\{l\in \Z_{>0} :(\la, l)\in \La\}$. Then the coefficients of its power series expansion $\mathcal{R}(t) = \sum_{m=0}^{\infty}R_{m} t^m$ takes the form
\begin{equation}\label{eqn:Rm}
R_{m} =\sum_{(\la,\zeta,l)\in P\times \C^{\times} \times \Z_{>0}} C(\mathcal{R},\la,\zeta,l)\zeta^m e^{m \la/l}
\end{equation}
for some $C(\mathcal{R},\la,\zeta,l)\in \K(e^{\omega_j/{l_0}})_{j\in I}$, which vanishes unless $\la\in \La$ and $\zeta^l = 1$. Here $\K(e^{\omega_j/{l_0}})_{j\in I}$ denotes the field extension of $\K$ obtained by adjoining $e^{\omega_j/{l_0}},\, j\in I$. The condition that $D(t)$ has only simple roots guarantees that each $C(\mathcal{R},\la,\zeta,l)$ is independent of $m$. Assume further that $R_m$ is $W$-invariant for all $m$. Here $W$ acts on $\K(e^{\omega_j/{l_0}})_{j\in I}$ by $w(e^{\omega_j/{l_0}}):=e^{w(\omega_j/{l_0})}$ for $w\in W$. In such a case, of course, $\mathcal{R}$ is $W$-invariant, and thus
\begin{equation}\label{eqn:RWsym}
w\left(C(\mathcal{R},\la,\zeta,l)\right)= C(\mathcal{R},w(\la),\zeta,l)
\end{equation}
for each $w\in W$; we will use this property in Proposition \ref{thm:algrel}.
\begin{definition}\label{def:anc}
Assume that $a\in I$ satisfies the assumption of Theorem \ref{thm:linQ}. From the above discussion there exists $C(\mathcal{Q}^{(a)},\la,\zeta,l)\in \K(e^{\omega_j/{t_a}})_{j\in I}$ for each $(\la,\zeta,l)\in P\times \C^{\times} \times \Z_{>0}$ such that
\begin{equation}\label{eq:QandC0}
\Qam =\sum_{(\la,\zeta,l)} C(\mathcal{Q}^{(a)},\la,\zeta,l)\zeta^m e^{m \la/l}, \qquad m\in \Z_{\geq 0},
\end{equation}
which vanishes unless either
\begin{itemize}
\item $(\la,\zeta,l)=(\la,1,1)$ with $\la\in \La_a$; or,
\item $(\la,\zeta,l)=(\la,\zeta,t_a)$ with $\la\in \La_a'$ and $\zeta^{t_a}=1$. 
\end{itemize}
Let $C^{(a)}_{\la}:=C(\mathcal{Q}^{(a)},\la,1,1)$ for $\la\in \La_a$ and $C^{(a)}_{\la,\zeta}:=C(\mathcal{Q}^{(a)},\la,\zeta,t_a)$ for $\la\in \La_a'$ and $\zeta^{t_a}=1$. We also put $\ncC{a}:=C^{(a)}_{\omega_a}$ to emphasize its close connection with $\nc{a}$; see Proposition \ref{prop:fixdomain}.
We can rewrite (\ref{eq:QandC0}) as
\begin{equation}\label{eq:QandC}
\Qam =\sum_{\la\in \La_a} C^{(a)}_{\la}e^{m \la}+\sum_{\la\in \La_a'}\sum_{\zeta : \zeta^{t_a}=1} C^{(a)}_{\la,\zeta}\zeta^m e^{m \la/t_a}.
\end{equation}
\end{definition}
\begin{remark}
Whenever we mention $C(\mathcal{Q}^{(a)},\la,\zeta,l)$, we need the assumption that Theorem \ref{thm:linQ} holds for $a\in I$. Otherwise, $C(\mathcal{Q}^{(a)},\la,\zeta,l)$ does not make any sense in contrast to $\nc{a}$, which is defined for any $a\in I$.
\end{remark}
\begin{prop}\label{prop:Cdomain}
Let $\ell_a$ be the cardinality of $t_a\La_{a}\cup \La'_{a}$ and $\la_i,\, 1\leq i\leq \ell_a$ be its distinct elements, i.e., $t_a\La_{a}\cup \La'_{a} = \{\la_1,\dots, \la_{\ell_a}\}$. Fix an integer $0\leq k\leq t_a-1$.
For $\la\in \La_a$,
$$
C^{(a)}_{\la}(e^{k\la})\prod_{1\leq i<j\leq \ell_a}(e^{\la_j}-e^{\la_i})\in \Z[P],
$$
and for $\la\in \La'_a$,
$$
\left(\sum_{\zeta :\zeta^{t_a}=1}C^{(a)}_{\la,\zeta}(\zeta e^{\la/t_a})^k\right) \prod_{1\leq i<j\leq \ell_a}(e^{\la_j}-e^{\la_i})\in \Z[P].
$$
\end{prop}
\begin{proof}
Consider (\ref{eq:QandC}) for $m=k,t_a+k,2t_a+k,\dots, (\ell_a-1)t_a+k$, i.e. the following system of $\ell_a$ linear equations :
$$
\begin{cases}
\begin{aligned}
\sum_{\la\in \La_a} C^{(a)}_{\la}(e^{k\la})(e^{t_a \la})^{0} & +\sum_{\la\in \La'_a} \left(\sum_{\zeta :\zeta^{t_a}=1}C^{(a)}_{\la,\zeta}(\zeta e^{\la/t_a})^k\right)(e^{\la})^{0} &  = & \Q{a}{0+k} \\
\sum_{\la\in \La_a} C^{(a)}_{\la}(e^{k\la})(e^{t_a \la})^{1}& +\sum_{\la\in \La'_a} \left(\sum_{\zeta :\zeta^{t_a}=1}C^{(a)}_{\la,\zeta}(\zeta e^{\la/t_a})^k\right) (e^{\la})^{1} &  = & \Q{a}{t_a+k} \\
\vdots \\
\sum_{\la\in \La_a} C^{(a)}_{\la}(e^{k\la})(e^{t_a \la})^{\ell_a-1}& +\sum_{\la\in \La'_a} \left(\sum_{\zeta :\zeta^{t_a}=1}C^{(a)}_{\la,\zeta}(\zeta e^{\la/t_a})^k\right) (e^{\la})^{\ell_a-1} &  = &\Q{a}{(\ell_a-1)t_a+k} \\
\end{aligned}
\end{cases}
$$
Since $e^{\la}\, (\la \in P)$ and $\Qam$ are in $\Z[P]$, Cramer's rule and Vandermonde determinant give 
$$
C^{(a)}_{\la}(e^{k\la}) \prod_{1\leq i<j\leq \ell_a}(e^{\la_j}-e^{\la_i})\in \Z[P],\, \la \in \La_a,
$$
and 
$$
\left(\sum_{\zeta :\zeta^{t_a}=1}C^{(a)}_{\la,\zeta}(\zeta e^{\la/t_a})^k\right)\prod_{1\leq i<j\leq \ell_a}(e^{\la_j}-e^{\la_i})\in \Z[P],\, \la \in \La'_a.
$$
\end{proof}
Let us now define a non-empty open subset of $\h_{\R}^{*}$ on which all $\nc{a}$ and $C(\mathcal{Q}^{(a)},\la,\zeta,l)$ 
can be regarded as well-defined functions. From Proposition \ref{prop:Cdomain}, it is necessary to consider  
the set $H_{\la_i,\la_j}:=\{\mu \in \h_{\R}^{*}: (\la_i-\la_j,\mu)=0\}$ for distinct $\la_i,\la_j\in t_a\La_{a}\cup \La'_{a}$. Note that $H_{\la_i,\la_j}$ is a hyperplane in $\h_{\R}^{*}$ passing through $0$. Consider the union $H_{a}:=\cup_{(\la_i, \la_j)} H_{\la_i,\la_j}$ of all such hyperplanes.
\begin{definition}\label{def:dom1}
Let $\dom_0$ be the non-empty open set in Proposition \ref{prop:ncfunc}. We define $\dom_1:= \dom_0\backslash (\cup_{a\in I} H_{a})$. Since $\cup_{a\in I} H_{a}$ is a finite union of hyperplanes, $\dom_1$ is a non-empty open set of $\h_{\R}^{*}$.
\end{definition}
\begin{cor}\label{cor:Qsimpleform}
As functions on $\dom_1$, $\nc{a}$ and $C(\mathcal{Q}^{(a)},\la,\zeta,l)$ are well-defined.
\end{cor}
\begin{proof}
For $\nc{a}$, this is a consequence of Proposition \ref{prop:ncfunc}. Now it is enough to consider non-zero $C(\mathcal{Q}^{(a)},\la,\zeta,l)$, namely, $C^{(a)}_{\la}$ and $C^{(a)}_{\la,\zeta}$. Applying Proposition \ref{prop:Cdomain} with $k=0$, we have $C^{(a)}_{\la}\prod_{1\leq i<j\leq \ell_a}(e^{\la_j}-e^{\la_i})\in \Z[P]$, and thus $C^{(a)}_{\la}$ is well-defined on $\dom_1$.
Finally, let $d_k:=\sum_{\zeta :\zeta^{t_a}=1}C^{(a)}_{\la,\zeta}(\zeta e^{\la/t_a})^k,\, k=0,\dots,t_a-1$, i.e.,
$$
\begin{bmatrix}
 (\zeta _0 e^{\frac{\lambda }{t_a}})^{0} & (\zeta _1 e^{\frac{\lambda }{t_a}})^{0} & \dots & (\zeta _{t_a-1} e^{\frac{\lambda }{t_a}})^{0} \\
 (\zeta _0 e^{\frac{\lambda }{t_a}})^{1} & (\zeta _1 e^{\frac{\lambda }{t_a}})^{1} & \dots & (\zeta _{t_a-1} e^{\frac{\lambda }{t_a}})^{1} \\
\vdots & \vdots & \vdots & \ddots &\vdots \\
 (\zeta _0 e^{\frac{\lambda }{t_a}})^{t_a-1} & (\zeta _1 e^{\frac{\lambda }{t_a}})^{t_a-1} & \dots & (\zeta _{t_a-1} e^{\frac{\lambda }{t_a}})^{t_a-1} \\
\end{bmatrix}
\begin{bmatrix}
C_{\la,\zeta_0} \\
C_{\la,\zeta_1} \\
C_{\la,\zeta_2} \\
\vdots \\
C_{\la,\zeta_{t_a-1}}
\end{bmatrix}
=
\begin{bmatrix}
d_0 \\
d_1 \\
d_2 \\
\vdots \\
d_{t_a-1}
\end{bmatrix}
$$
where $\zeta_j = e^{2\pi i j/t_a}\in \C^{\times}$. By Proposition \ref{prop:Cdomain}, $d_j\prod_{1\leq i<j\leq \ell_a}(e^{\la_j}-e^{\la_i})\in \Z[P]$. The determinant of the square matrix above is $e^{\la}$ multiplied by a non-zero complex number. By inverting the matrix, we see that $C^{(a)}_{\la,\zeta}$ is well-defined on $\dom_1$ as each $d_k$ is well-defined on $\dom_1$ by Proposition \ref{prop:Cdomain}.
\end{proof}
\begin{prop}\label{prop:fixdomain}
As functions on $\dom_1$, $\ncC{a}=\nc{a}$, and they assume only positive real values.
\end{prop}
\begin{proof}
Consider (\ref{eq:QandC}). The asymptotic behavior of $\Q{a}{m}$, as functions on $\dom_1$, is governed by $e^{m\omega_a}$ as $m\to \infty$ :
$$
\lim_{m\to \infty} e^{-m\omega_a}\Q{a}{m} = 
\lim_{m\to \infty} \left(\sum_{\la\in \La_a} C^{(a)}_{\la}e^{m(-\omega_a+\la)}+\sum_{\la\in \La_a'}\sum_{\zeta : \zeta^{t_a}=1} C^{(a)}_{\la,\zeta}\zeta^m e^{m(-\omega_a+ \la/t_a)}\right) = \ncC{a}
$$
by \ref{thm:lineq1} and \ref{thm:lineq2} of Theorem \ref{thm:linQ}. We are also using 
Corollary \ref{cor:Qsimpleform} to make sure that all the expressions make sense on $\dom_1$. Combining this with Proposition \ref{prop:ncfunc}, we can conclude that $\ncC{a}=\nc{a}$ on $\dom_1$. 

The last statement follows from Proposition \ref{prop:Epoly}, which implies that
$$\ncQ{a}{m} = 1+e^{-\alpha_{a}}E_0+\dots + e^{-m\alpha_{a}}E_{m-1}\in 1+e^{-\alpha_{a}}\Z_{\geq 0}[e^{-\alpha_j}]_{j\in I}$$
and so $\ncQ{a}{m}\geq 1$ on $\dom_0\supseteq \dom_1$. Therefore, $\ncC{a}=\nc{a}\geq 1$.
\end{proof}
Hence, we can think of the coefficients $C(\mathcal{Q}^{(a)},\la,\zeta,l)$ as analogues of $\nc{a}$. Apart from $C(\mathcal{Q}^{(a)},\omega_a,1,1)=\ncC{a}$, which is equal to $\nc{a}$ and $C(\mathcal{Q}^{(a)},\omega_a-\alpha_a,1,1)$ which is closely related to the $\widetilde{Q}$-variable in \cite{2016arXiv160605301F}, such coefficients have not been defined and considered as objects of study in their own right. Their existence has become clearly visible after Theorem \ref{thm:linQ}.
\begin{example}
Let us consider an example. In type $A_1$, we have
$$
\ncC{1} = \frac{e^{\omega_1}}{e^{\omega_1}-e^{-\omega_1}} \in \K = \C(e^{\omega_1}),
$$
and
$$
\nc{1}=1+e^{-\alpha_1}+e^{-2\alpha_1}+\dots\in \Z[[e^{-\alpha_1}]].
$$
In this case, we can easily see that the power series $\nc{1}$ can be summed to be an element in $\K$. However, we do not know this in advance in general. \cite[Theorem 6.4]{MR2982441} gives an explicit formula for $\nc{a}$ but it still does not imply that $\nc{a}\in \K$. That is why we make the identification of two elements as functions on $\dom_1$ as in Proposition \ref{prop:fixdomain}. After proving Theorem \ref{thm:main}, we no longer need to worry much about their membership as we can regard $\ncMY{a}$ as elements of both $\K$ and $\Z[[e^{-\alpha_j}]]_{j\in I}$ by definition.
\end{example}

\subsection{Polyhedral formula for decomposition of KR modules}\label{ss:polyform}
Recall that $\theta$ denotes the highest root of $\g$. In \cite{MR1836791} Chari proved some explicit formula for decomposition of KR modules. For given $a\in I$, assume that $[\theta]_a\leq 2$. Then there exist a tuple of positive integers $(b_j)_{j\in J_{a}}$, and a tuple of dominant integral weights $(\la_j)_{j\in J_{a}}$ for some finite index set $J_{a}$ such that
\begin{equation}\label{eqn:KRpolyform}
\rwma=\bigoplus_{\mathbf{x}\in F_m^{(a)}} L(\la_{\mathbf{x}})
\end{equation}
where 
$F_m^{(a)}=\{(x_j)_{j\in J_{a}}\mid \sum_{j\in J_{a}}b_jx_j=m,\, x_j\in \Z_{\geq 0}\}$, and $\la_{\mathbf{x}}=\sum_{j\in J_{a}}x_j\la_j$ for each $\mathbf{x}\in F_m^{(a)}$. We call (\ref{eqn:KRpolyform}) a polyhedral formula. We will use some polyhedral formulas in Section \ref{sec:MYandpoly}. See \cite{MR1745263} for a systematic exposition on the subject and also for conjectural polyhedral formulas with multiplicity in exceptional types.

\section{Algebraic relations among normalized characters}\label{sec:algrelnc}
In this Section, we study (\ref{eqn:norsys}) for $\ncCfam$ and $\ncMYfam$, respectively. The central results are Propositions \ref{thm:algrel} and \ref{prop:relforMY}.
\begin{remark}
Throughout the paper, we never use an expression $s_a(\nc{a})$ in which $\nc{a}$ is regarded as a power series in $\Z[[e^{-\alpha_j}]]_{j\in I}$. We have an action of $W$ on $\K$ but not on $\Z[[e^{-\alpha_j}]]_{j\in I}$. That is why we do not study
(\ref{eqn:norsys}) for $\ncfam$ directly as a relation among power series. After establishing (\ref{eqn:norsys}) for $\ncMYfam$ and the identity $\nc{a}=\ncMY{a}$, we consequently have (\ref{eqn:norsys}) for $\ncfam$.
\end{remark}
\begin{prop}\label{prop:norR}
Let $a\in I$ be arbitrary. On the domain $\dom_0$ (see Proposition \ref{prop:ncfunc}), the limit
$$
\lim_{m\to \infty} e^{-m (\omega_a+s_a(\omega_a))}\left((\Qam)^2 - Q^{(a)}_{m+1}Q^{(a)}_{m-1}\right)
$$
exists, and it is equal to $\prod_{b:C_{ab}<0}(\nc{b})^{-C_{ab}}$.
\end{prop}
\begin{proof}
Note that this is equivalent to
\begin{equation}\label{eq:norR}
\lim_{m\to \infty} e^{m(\sum_{b:C_{ab}<0} C_{ba}\omega_b)}\prod_{b : C_{ab}< 0} \prod_{k=0}^{-C_{a b}-1}
Q^{(b)}_{\bigl \lfloor\frac{C_{b a}m - k}{C_{a b}}\bigr \rfloor} =\prod_{b : C_{ab}< 0} \left(\nc{b}\right)^{-C_{a b}}
\end{equation}
by (\ref{eq:Qsys}) and the fact that $\omega_a+s_a(\omega_a)= 2\omega_a-\alpha_a = -\sum_{b:C_{ab}<0} C_{ba}\omega_b$.

For $b\in I$ such that $C_{ab}<0$ and an integer $k$ with $0\leq k\leq -C_{a b}-1$, we have
$$
\lim_{m\to \infty} e^{\bigl \lfloor\frac{C_{b a}m - k}{C_{a b}}\bigr \rfloor \omega_a}
Q^{(b)}_{\bigl \lfloor\frac{C_{b a}m - k}{C_{a b}}\bigr \rfloor} = \nc{b}.
$$
Using Hermite's identity 
$$
\sum_{k=0}^{-C_{a b}-1} \bigl \lfloor\frac{-C_{b a}m + k}{-C_{a b}}\bigr \rfloor = -C_{ba}m, \qquad m\in \Z_{\geq 0}
$$
we obtain
$$
\lim_{m\to \infty} e^{mC_{ba}\omega_b}\prod_{k=0}^{-C_{a b}-1}
Q^{(b)}_{\bigl \lfloor\frac{C_{b a}m - k}{C_{a b}}\bigr \rfloor} = \left(\nc{b}\right)^{-C_{a b}},
$$
which implies (\ref{eq:norR}).
\end{proof}

\begin{lemma}\label{lem:wtineq}
For each $a\in I$, we have the following :
\begin{enumerate}[label=(\roman*), ref=(\roman*)]
\item \label{thm:main:6} For each $\la\in \La_a$ not equal to $\omega_a$, $s_a(\omega_a)\geq \la$.
\item \label{thm:main:7} For each $\la\in \La_a'$, $t_as_a(\omega_a)\geq \la$.
\end{enumerate}
\end{lemma}
\begin{proof}
Let us denote the set of weight of $L(\omega_a)$ by $\Omega\left(L(\omega_a)\right)$. For an explicit description of $\La_a$ and $\La'_a$, see the Appendix of \cite{Lee2016}.

Consider \ref{thm:main:6}. Note first that $\La_a$ is a subset of $\Omega\left(L(\omega_a)\right)$. We can actually prove a stronger statement that $s_a(\omega_a)\succeq \la$ for any $\la \in \Omega\left(L(\omega_a)\right),\, \la\neq \omega_a$. Note that $\Omega\left(L(\omega_a)\right)$ is a saturated set of weights with highest weight $\omega_a$ in the sense of \cite[Section 13.4]{MR499562}. Then the argument of \cite[Lemma 13.4 B]{MR499562} shows that for any weight $\la \in \Omega\left(L(\omega_a)\right)$ such that $\la \neq \omega_a$, there exists a simple root $\alpha$ such that $\la+\alpha\in \Omega\left(L(\omega_a)\right)$. Repeated application of this implies that $\omega_a-\alpha_b \succeq \la$ for some $b\in I$. However, $\omega_a-\alpha_b \in \Omega\left(L(\omega_a)\right),\, b\in I$ if and only if $b=a$. This implies $s_a(\omega_a) = \omega_a-\alpha_a \succeq \la$.

Now we turn to \ref{thm:main:7}. We only have to consider it when $t_a\neq 1$, otherwise $\La_a'$ is empty. The proof is similar to the proof of \ref{thm:lineq2} of Theorem \ref{thm:linQ}. We need to check the following:

In type $B_r$, $2(\omega_r-\alpha_r)-\la\geq 0$ for $\la\in \La'_r$; it follows from $2(\omega_r-\alpha_r)-\omega_{r-2}=\alpha_{r-1}\geq 0$ and $\omega_{r-2}-\la\geq 0$ for $\la \in \La'_r=\La_{r-2}$.

In type $C_r$, $2(\omega_a-\alpha_a)-\la\geq 0$ for $\la\in \{\omega_0,\omega_1,\dots,\omega_{a-1}\}$ for each $a\in I\backslash{\{r\}}$; writing 
$2(\omega_a-\alpha_a) -\la = (\omega_a-\alpha_a)+(\omega_a-\la)-\alpha_a$, 
we can use $\omega_a\geq \alpha_a$ for each $a\in I$ and 
$$\omega_j-\omega_{j-1}=\frac{\alpha_r}{2}+\sum_{i=j}^{r-1}\alpha_i \geq 0, \forall j\in I\backslash{\{r\}},$$
which implies $(\omega_a-\omega_j)-\alpha_a\geq 0$ for $j=0,1,\dots, a-1$.

In type $F_4$, $2(\omega_3-\alpha_3)-\la\geq 0$ for $\la \in \La'_3\cap P^+=\{0,\omega _1,\omega _2,2 \omega _1,2 \omega _4,\omega _1+2 \omega _4\}$ and $2(\omega_4-\alpha_4)-\la\geq 0$ for $\la \in \{0,\omega_1\}=\La'_4\cap P^+=\{0,\omega _1\}$; in \cite[Proposition 3.1]{Lee2016}, we have already shown that $2\omega_3-\la \geq 2\alpha_3$ for $\la \in \La'_3\cap P^+$ and $2\omega_4-\la \geq 2\alpha_4$ for $\la \in \La'_4\cap P^+$.

In type $G_2$, $3(\omega_2-\alpha_2)-\la\geq 0$ where $\la \in \{0,\omega_1\}=\La'_2\cap P^{+}$; $3(\omega_2-\alpha _2)-0 =3\alpha _1+3\alpha _2$ and $3(\omega_2-\alpha _2)-\omega_1 =\alpha _1$.
\end{proof}

\begin{prop}\label{thm:algrel}
Assume that Theorem \ref{thm:linQ} holds for $a\in I$. Then
$$
(1-e^{\alpha_a})(1-e^{-\alpha_a})\ncC{a}s_a(\ncC{a}) = \prod_{b:C_{ab}<0}(\nc{b})^{-C_{ab}}
$$
as functions on $\dom_1$.
\end{prop}

\begin{proof}
Let us plug (\ref{eq:QandC0}) into the left-hand side of (\ref{eq:Qsys}), i.e. $(\Qam)^2 - Q^{(a)}_{m+1}Q^{(a)}_{m-1}$ :
$$
\sum_{(\la_1,\zeta_1,l_1),(\la_2,\zeta_2,l_2)} C(\mathcal{Q}^{(a)},\la_1,\zeta_1,l_1)C(\mathcal{Q}^{(a)},\la_2,\zeta_2,l_2)(1-\frac{\zeta_1 e^{\la_1/l_1}}{\zeta_2 e^{\la_2/l_2}})\left((\zeta_1\zeta_2)^m e^{m (\la_1/l_1+\la_2/l_2)}\right).
$$ 
Hence, the coefficient for the exponential term $(\zeta_1\zeta_2)^m e^{m (\la_1/l_1+\la_2/l_2)}$ is
$$
C(\mathcal{Q}^{(a)},\la_1,\zeta_1,l_1)C(\mathcal{Q}^{(a)},\la_2,\zeta_2,l_2)(2-\frac{\zeta_1 e^{\la_1/l_1}}{\zeta_2 e^{\la_2/l_2}}-\frac{\zeta_2 e^{\la_2/l_2}}{\zeta_1 e^{\la_1/l_1}}),
$$
which vanishes when $(\la_1,\zeta_1,l_1)=(\la_2,\zeta_2,l_2)$.
Thus, the term $e^{m\omega_a+m\omega_a}$ does not appear in the above sum. Lemma \ref{lem:wtineq} shows that the exponential term $e^{m (\omega_a+s_a(\omega_a))}$ is dominating :
\begin{equation}\label{eq:norL}
\begin{aligned}
\lim_{m\to \infty} e^{-m (\omega_a+s_a(\omega_a))}\left((\Qam)^2 - Q^{(a)}_{m+1}Q^{(a)}_{m-1}\right) & = C^{(a)}_{\omega_a}C^{(a)}_{s_a(\omega_a)}(2-\frac{e^{\omega_a-s_a(\omega_a)}}{e^{s_a(\omega_a)-\omega_a}}-\frac{e^{s_a(\omega_a)-\omega_a}}{e^{\omega_a-s_a(\omega_a)}})\\
& =\ncC{a}s_a(\ncC{a})(1-e^{\alpha_a})(1-e^{-\alpha_a}).
\end{aligned}
\end{equation}
We have used (\ref{eqn:RWsym}) in the last line to get $s_a(\ncC{a}) = s_a(C^{(a)}_{\omega_a}) = C^{(a)}_{s_a(\omega_a)}$. Also note that we need 
Corollary \ref{cor:Qsimpleform} in this limiting procedure as in Proposition \ref{prop:fixdomain}. By Proposition \ref{prop:norR} we obtain the desired conclusion.
\end{proof}

\begin{cor}\label{cor:algrel2}
Let $\g$ be a simple Lie algebra which is not of type $E_7$ or $E_8$. The family $X=\ncCfam$ satisfies (\ref{eqn:norsys}) for every $a\in I$.
\end{cor}
\begin{proof}
For such $\g$, Theorem \ref{thm:linQ} holds for any $a\in I$. By Proposition \ref{prop:fixdomain}, $\nc{a}=\ncC{a}$ as functions on $\dom_1$. We can finish the proof by Proposition \ref{thm:algrel}.
\end{proof}
\begin{remark}\label{rmk:QQcomparison}
Let us compare (\ref{eqn:norsys}) with the $Q\widetilde{Q}$-system \cite[Theorem 3.2]{2016arXiv160605301F}. If we turn the $Q\widetilde{Q}$-system into a relation among ordinary characters (not $q$-characters so that we can ignore the spectral parameters), then it takes the following form
\begin{equation}\label{eqn:QQrel}
e^{\alpha_i/2}Q_{a}\widetilde{Q}_{a}-e^{-\alpha_i/2}Q_{a}\widetilde{Q}_{a} = \prod_{b:C_{ab}<0}(Q_{b})^{-C_{ab}},
\end{equation}
which is satisfied by setting $Q_{a} = \nc{a}$ and 
$$
\widetilde{Q}_{a} = \frac{e^{-\alpha_i/2}}{1-e^{-\alpha_i}}(\nc{a})^{-1}\prod_{b:C_{ab}<0}(\nc{b})^{-C_{ab}},\qquad \text{\cite[Proposition 4.7]{2016arXiv160605301F}}.
$$
To be able to obtain (\ref{eqn:norsys}) from (\ref{eqn:QQrel}), we need an additional relation $\widetilde{Q}_{a} = (e^{-\alpha_a/2}-e^{\alpha_a/2})s_a(Q_a)$, which is not addressed in \cite{2016arXiv160605301F}. It will be an interesting problem to find its $q$-character analogue, which might enable a more elegant approach to (\ref{eqn:norsys}) than given here without relying on Theorem \ref{thm:linQ}.
\end{remark}

Now we consider (\ref{eqn:norsys}) for $\ncMYfam$.
\begin{lemma}\label{lma:simpleref}
Let $\alpha$ be a root of $\g$ and $a\in I$. Then
$$
[\alpha]_a+[s_a(\alpha)]_{a} = -\sum_{b:C_{ab}<0}C_{ab}[\alpha]_b.
$$
\end{lemma}

\begin{proof}
Let $c_b := [\alpha]_b$. Then $\alpha = \sum_{b\in I}c_b\alpha_b$. Note that 
$$
\begin{aligned}
s_a(\alpha) & =\sum_{b\in I}c_bs_a(\alpha_b) =\sum_{b\in I}c_b\left(\alpha_b - \alpha_b(h_a)\alpha_a\right) \\
& =\sum_{b\in I}c_b\alpha_{b}-\sum_{b\in I}C_{ab}c_b\alpha_a,
\end{aligned}
$$
which implies
$$
\left[\alpha\right]_a+[s_a(\alpha)]_{a} = c_a+(c_a- \sum_{b\in I}C_{ab}c_b)  = -\sum_{b:C_{ab}<0}C_{ab}[\alpha]_b.
$$
\end{proof}

\begin{prop}\label{prop:relforMY}
The family $X=\left(\ncMY{a}\right)_{a\in I}$ satisfies (\ref{eqn:norsys}) for every $a\in I$.
\end{prop}
\begin{proof}
Let us consider the left-hand side of $(\ref{eqn:norsys})$. Recall the well-known fact that for any positive root $\alpha\neq \alpha_a$, $s_a(\alpha)$ is again a positive root. Hence,
$$
\begin{aligned}
(1-e^{\alpha_a})(1-e^{-\alpha_a})\ncMY{a}s_a(\ncMY{a}) & = \frac{1}{\prod_{\alpha\in \posR\backslash{\{\alpha_a\}}}(1-e^{-\alpha})^{[\alpha]_a+[s_a(\alpha)]_{a}}} \\
& = \frac{1}{\prod_{\alpha\in \posR}(1-e^{-\alpha})^{[\alpha]_a+[s_a(\alpha)]_{a}}} & \text{as $[\alpha_a]_a+[s_a(\alpha_a)]_{a}=0$}.
\end{aligned}
$$

The right-hand side of (\ref{eqn:norsys}) is
$$
\prod_{b:C_{ab}<0}(\ncMY{b})^{-C_{ab}} = \prod_{b:C_{ab}<0}\frac{1}{\prod_{\alpha\in \posR}(1-e^{-\alpha})^{-C_{ab}[\alpha]_b}} = \frac{1}{\prod_{\alpha\in \posR}(1-e^{-\alpha})^{-\sum_{b:C_{ab}<0} C_{ab}[\alpha]_b}}.
$$

We can now finish the proof by Lemma \ref{lma:simpleref}.
\end{proof}

\section{Proof of Main Theorem : preliminary cases}\label{sec:MYandpoly}
In this Section, we prove (\ref{eqn:MY}) for some $a\in I$. Here we can use the corresponding polyhedral formula without much effort. Let us recall the Weyl denominator formula
$$
\sum_{w\in W}(-1)^{\ell(w)}e^{w(\rho)} = e^{\rho}\weylD.
$$
More generally, for a subset $J\subseteq I$, applying the above, we obtain
$$
\sum_{w\in W_{J}}(-1)^{\ell(w)}e^{w(\rho)} = e^{\rho}\prod_{\alpha \in \posR(J)}(1-e^{-\alpha}).
$$
In particular, when $J=I\backslash \{a\}$ for some $a\in I$, $W_{J} = W_{\omega_a}$ and the above equality takes the form
\begin{equation}\label{eq:para2}
\sum_{w\in W_{\omega_a}}(-1)^{\ell(w)}e^{w(\rho)} = e^{\rho}\prod_{\alpha \in \posR(I\backslash \{a\})}(1-e^{-\alpha}).
\end{equation}
\begin{prop}\label{prop:type1}
Let $\g$ be a simple Lie algebra and assume that $a\in I$ satisfies $[\theta]_{a}=1$. Then $\nc{a} = \ncC{a} = \ncMY{a}$ holds.
\end{prop}
\begin{proof}
By Proposition \ref{prop:fixdomain}, it is enough to show $\ncC{a} = \ncMY{a}$. In this case, the polyhedral formula takes the simplest possible form, namely, $\Qam = \chi(L(m\omega_a))$. So the Weyl character formula implies
$$
\Qam  = \frac{\sum_{w\in W}(-1)^{\ell(w)}e^{w(\rho)}e^{w(m\omega_a)}}{e^{\rho}\weylD}.
$$
Hence,
$$
\begin{aligned}
\ncC{a} & = \frac{\sum_{w\in W_{\omega_a}}(-1)^{\ell(w)}e^{w(\rho)}}{e^{\rho}\weylD} & \\
& = \frac{\prod_{\alpha \in \posR(I\backslash \{a\})}(1-e^{-\alpha})}{\weylD} & \text{by (\ref{eq:para2})} \\
& = \frac{1}{\prod_{\alpha \in \posR,\, [\alpha]_a\neq 0}(1-e^{-\alpha})}.
\end{aligned}
$$
The fact that $\theta[a]=1$ implies $[\alpha]_a \leq 1$ for each $\alpha\in \posR$. Therefore, the last line is equal to $\ncMY{a}$.
\end{proof}
\begin{prop}\label{prop:type2}
Assume that $(\g,a)$ is one of the following pairs  :
$$
(E_7,1), (E_8,7), (F_4,1), (G_2,1). 
$$
Then $\nc{a} = \ncC{a} = \ncMY{a}$ holds.
\end{prop}
\begin{proof}
First we observe the following properties :
\begin{itemize}
\item $[\theta]_{a}=2$; 
\item $\omega_a = \theta$ and hence,  $\omega_a \in \posR$;
\item $\omega_a$ is the unique positive root $\alpha\in \posR$ such that $[\alpha]_{a}=2$.
\end{itemize}
In this case, the polyhedral formula is $\Qam = \sum_{k=0}^{m}\chi(L(k\omega_a))$, giving
$$
\Qam  = \frac{\sum_{k=0}^{m}\sum_{w\in W}(-1)^{\ell(w)}e^{w(\rho)}e^{w(k\omega_a)}}{e^{\rho} \weylD}.
$$
Hence,
$$
\begin{aligned}
\ncC{a} & = \frac{\sum_{w\in W_{\omega_a}}(-1)^{\ell(w)}e^{w(\rho)}\frac{1}{(1-e^{-\omega_a})}}{e^{\rho}\weylD} \\
& = \frac{1}{(1-e^{-\omega_a})}\frac{\prod_{\alpha \in \posR(I\backslash \{a\})}(1-e^{-\alpha})}{\weylD} & \text{by (\ref{eq:para2})} \\
& = \frac{1}{(1-e^{-\omega_a})}\frac{1}{\prod_{\alpha \in \posR,\, [\alpha]_a\neq 0}(1-e^{-\alpha})}.
\end{aligned}
$$
From the observation that $\omega_a=\theta\in \posR$ and it is the unique positive root with $[\alpha]_{a} = 2$, we conclude $\ncC{a}=\ncMY{a}$ and thus, $\nc{a}=\ncMY{a}$ by Proposition \ref{prop:fixdomain}.
\end{proof}

\begin{prop}\label{prop:type3E8}
Assume that $\g$ is of type $E_8$ and $a=1$. Then $\nc{1} = \ncC{1} = \ncMY{1}$.
\end{prop}
\begin{proof}
By Proposition \ref{prop:fixdomain}, it is enough to show $\ncC{1} = \ncMY{1}$.
For $a=1$, the polyhedral formula is given by
$\Q{1}{m}=\sum_{\mathbf{x}\in F_m^{(1)}} \chi(L(\la_{\mathbf{x}}))$, where 
$F_m^{(1)}=\{(x_0,x_1,x_2) : x_0+x_1+x_2=m,\, x_j\in \Z_{\geq 0}\}$, and $\la_{\x}=x_1\omega_1+x_2\omega_7$ for $\x=(x_0,x_1,x_2)\in F_m^{(1)}$.
From this, we get
$$
\nc{1} = \frac{\sum_{w\in W_{\omega_1}}(-1)^{\ell(w)}e^{w(\rho)}\frac{1}{(1-e^{-w(\omega_1-\omega_7)})}}{(1-e^{-\omega_1})e^{\rho}\weylD}
$$
Let $J(D_7) :  = \{2,3,4,5,6,7,8\}$ and $J(D_6) : = \{2,3,4,5,6,8\}$. Then
$$
\begin{aligned}
S_1 : = \sum_{w\in W_{\omega_1}}(-1)^{\ell(w)}e^{w(\rho)}\frac{1}{(1-e^{-w(\omega_1-\omega_7)})} & = 
\sum_{w\in W_{J(D_7)}/W_{J(D_6)}}(-1)^{\ell(w)}\frac{w\left(\sum_{w'\in W_{J(D_6)}}(-1)^{\ell(w')}e^{w'(\rho)}\right) }{(1-e^{-w(\omega_1-\omega_7)})}\\
& = 
\sum_{w\in W_{J(D_7)}/W_{J(D_6)}}(-1)^{\ell(w)}\frac{w\left(e^{\rho}\prod_{\alpha \in \posR(J(D_6))}(1-e^{-\alpha})\right) }{(1-e^{-w(\omega_1-\omega_7)})}.
\end{aligned}
$$
Hence, $\ncC{1} e^{\rho}\weylD = \ncMY{1}e^{\rho}\weylD$ is equivalent to
$$
\frac{S_1}{(1-e^{-\omega_1})} = \frac{e^{\rho}\weylD}{\MYprod{1}}=
\frac{e^{\rho}\prod_{\alpha \in \posR,\, [\alpha]_1=0}(1-e^{-\alpha})}{\prod_{\alpha \in \posR,\, [\alpha]_1=2}(1-e^{-\alpha})},
$$
where the last equality follows from $[\theta]_1=2$. We may use computer algebra systems to check the following equivalent identity:
$$
S_1 \prod_{\alpha \in \posR,\, [\alpha]_1=2}(1-e^{-\alpha}) = (1-e^{-\omega_1}) \times e^{\rho}\prod_{\alpha \in \posR,\, [\alpha]_1=0}(1-e^{-\alpha}).
$$
In fact, both sides are polynomials in $\Z[\alpha_j]_{j\in I}$; see \cite{GHMYF2018} for a computer implementation.
\end{proof}

\begin{prop}\label{prop:type3F4}
Assume that $\g$ is of type $F_4$ and $a=4$. Then (\ref{eqn:main}) holds.
\end{prop}
\begin{proof}
Again, it is enough to show $\ncC{4} = \ncMY{4}$ by Proposition.
For $a=4$, the polyhedral formula is 
$
\Q{4}{m}=\sum_{\mathbf{x}\in F_m^{(4)}} \chi(L(\la_{\mathbf{x}}))
$
where 
$F_m^{(4)}=\{(x_0,x_1,x_2) : 2x_0+2x_1+x_2=m,\, x_j\in \Z_{\geq 0}\}$, and $\la_{\mathbf{x}}=x_1\omega_1+x_2\omega_4$ for $\mathbf{x}\in F_m^{(4)}$.
From this, we obtain
$$
\nc{4} = \frac{\sum_{w\in W_{\omega_4}}(-1)^{\ell(w)}e^{w(\rho)}\frac{1}{(1-e^{-w(2\omega_4-\omega_1)})}}{(1-e^{-2\omega_4})e^{\rho}\weylD}.
$$

Let
$$
\begin{aligned}
S_4 : = \sum_{w\in W_{\omega_4}}(-1)^{\ell(w)}e^{w(\rho)}\frac{1}{(1-e^{-w(2\omega_4-\omega_1)})} & = 
\sum_{w\in W_{\{1,2,3\}}/W_{\{2,3\}}}(-1)^{\ell(w)}\frac{w\left(\sum_{w'\in W_{\{2,3\}}}(-1)^{\ell(w')}e^{w'(\rho)}\right) }{(1-e^{-w(2\omega_4-\omega_1)})}\\
& = 
\sum_{w\in W_{\{1,2,3\}}/W_{\{2,3\}}}(-1)^{\ell(w)}\frac{w\left(e^{\rho}\prod_{\alpha \in \posR(\{2,3\})}(1-e^{-\alpha})\right) }{(1-e^{-w(2\omega_4-\omega_1)})}.
\end{aligned}
$$
As in the proof of Proposition \ref{prop:type3E8}, $\nc{4} = \ncMY{4}$ is equivalent to 
$$
S_4\times \prod_{\alpha \in \posR,\, [\alpha]_4=2}(1-e^{-\alpha}) = (1-e^{-2\omega_4}) \times e^{\rho}\prod_{\alpha \in \posR,\, [\alpha]_4=0}(1-e^{-\alpha})
$$
for which we may use computer algebra systems to check the equality; see \cite{GHMYF2018} for an implementation.
\end{proof}

\section{Proof of Main Theorem : general cases}\label{sec:mainproof}
Now we finish the proof of Theorem \ref{thm:main}.
\begin{proof}
Whenever $\ncC{a}$ is defined (i.e. Theorem \ref{thm:linQ} holds for that $a\in I$), $\ncC{a} = \nc{a}$ by Proposition \ref{prop:fixdomain}. Thus, it is sufficient to prove $\ncC{a}=\ncMY{a}$ when $\g$ is not either of type $E_7$ or $E_8$. To show this, we use the fact that both $\ncCfam$ and $\ncMYfam$ satisfy (\ref{eqn:norsys}) for any $a\in I$ (Corollary \ref{cor:algrel2} and Propositions \ref{prop:relforMY}). Note also that all $\ncC{a}$, $\nc{a}$ and $\ncMY{a}$ take only non-zero real values on $\dom_1$; by definition for $\ncMY{a}$ and by Proposition \ref{prop:fixdomain} for $\ncC{a}$ and $\nc{a}$. We clarify that we are considering (\ref{eqn:norsys}) as a relation among non-vanishing functions on $\dom_1$.

Let us give a brief argument for each type.  

When $\g$ is of type $A_r$, Proposition \ref{prop:type1} implies that $\nc{a} = \ncMY{a}$ for all $a\in I$ as $[\theta]_a=1$.

When $\g$ is of type $B_r\, (r\geq 2)$, $[\theta]_1=1$ and $[\theta]_a=2$ for $2\leq a\leq r$. Proposition \ref{prop:type1} applies for $a=1$ and thus, $\ncC{1} = \ncMY{1}$.
For $1\leq a\leq r-1$, (\ref{eqn:norsys}) takes the form
\begin{equation}\label{eqn:induction}
\Xrel{a} = \X{a-1}\X{a+1},\qquad \X{0}=1
\end{equation}
which clearly shows that any non-zero $X=\Xfam$ satisfying (\ref{eqn:induction}) is completely determined by $\X{1}$. Since both $\ncCfam$ and $\ncMYfam$ satisfy (\ref{eqn:induction}) and $\ncC{1} = \ncMY{1}$, $\ncC{a}=\ncMY{a}$ for all $a\in I$.

When $\g$ is of type $C_r\, (r\geq 3)$, $[\theta]_a=2$ for $1\leq a\leq r-1$ and $[\theta]_r=1$. Proposition \ref{prop:type1} implies that $\ncC{r}= \ncMY{r}$. For $a\in \{r-1,r\}$, (\ref{eqn:norsys}) gives
$$
\begin{cases} 
\Xrel{r} = \X{r-1} \\
\Xrel{r-1} = \X{r-2}(\X{r})^2.
\end{cases}
$$
For any non-zero $X=\Xfam$ satisfying the above, $\X{r-1}$ and $\X{r-2}$ are uniquely determined by $\X{r}$. For $2\leq a\leq r-2$, (\ref{eqn:norsys}) gives the same relation as (\ref{eqn:induction}) by which we can write $\X{a}$ with $1\leq a\leq r-3$ entirely in terms of $\X{r-2}$ and in turn, in terms of $\X{r}$. This proves $\ncC{a} = \ncMY{a}$ for all $a\in I$.

When $\g$ is of type $D_r\, (r\geq 4)$, $[\theta]_a=1$ for $a\in \{1,r-1,r\}$ and $[\theta]_a=2$ for $2\leq a\leq r-2$. From Proposition \ref{prop:type1} we have $\ncC{a}= \ncMY{a}$ for $a\in \{1,r-1,r\}$. Now (\ref{eqn:norsys}) with $1\leq a\leq r-3$ gives the same relation as (\ref{eqn:induction}) by which we can write $\X{a}$ with $1\leq a\leq r-2$ in terms of $\X{1}$ for non-zero $X=\Xfam$. This proves $\ncC{a} = \ncMY{a}$ for all $a\in I$.

When $\g$ is of type $E_6$, $([\theta]_1,\dots, [\theta]_6)=(1,2,3,2,1,2)$. Proposition \ref{prop:type1} implies that $\ncC{a} = \ncMY{a}$ for $a\in \{1,5\}$. Let us write (\ref{eqn:norsys}) for $a\in \{1,2,5,3\}$ :
$$
\begin{cases} 
\Xrel{1} = \X{2} \\
\Xrel{2} = \X{1}\X{3} \\
\Xrel{5} = \X{4} \\
\Xrel{3} = \X{2}\X{4}\X{6}.
\end{cases}
$$
A non-zero solution $\Xfam$ of this system is uniquely determined by $\X{1}$ and $\X{5}$, which proves $\ncC{a} = \ncMY{a}$ for all $a\in \{1,\dots,6\}$.

When $\g$ is of type $E_7$, $([\theta]_1,\dots, [\theta]_7)=(2, 3, 4, 3, 2, 1, 2)$. We need some care since 
Corollary \ref{cor:algrel2} is not available. Instead, we use Proposition \ref{thm:algrel} directly. From Propositions \ref{prop:type1} and \ref{prop:type2}, $\ncC{a}= \ncMY{a}$ for $a\in \{1,6\}$. Write (\ref{eqn:norsys}) for $a\in \{1,2,6,5,3\}$ :
$$
\begin{cases} 
\Xrel{1} = \X{2} \\
\Xrel{2} = \X{1}\X{3} \\
\Xrel{6} = \X{5} \\
\Xrel{5} = \X{4}\X{6} \\
\Xrel{3} = \X{2}\X{4}\X{7}.
\end{cases}
$$
We can observe that a non-zero solution $\Xfam$ of the above system is uniquely determined by $\X{1}$ and $\X{6}$. If we set 
$$
X_0=(\X{1},\dots,\X{7})=(\ncC{1},\ncC{2},\ncC{3},\nc{4},\ncC{5},\ncC{6},\ncC{7}),
$$
then $X_0$ satisfies the above system of equation by Proposition \ref{thm:algrel} (note again that $\ncC{4}$ cannot even be defined). Hence, $X_0 = \ncMYfam$ for all $a\in \{1,\dots,7\}$. Therefore, $\nc{a} = \ncMY{a}$ for all $a\in \{1,\dots,7\}$.

When $\g$ is of type $E_8$, $([\theta]_1,\dots, [\theta]_8)=(2, 4, 6, 5, 4, 3, 2, 3)$. Proposition \ref{thm:algrel} applies for $a=\{1,2,6,7\}$, and thus we have (\ref{eqn:norsys}) only for these nodes. From Propositions \ref{prop:type2} and \ref{prop:type3E8}, $\ncC{a}= \ncMY{a}$ for $a\in \{1,7\}$. Write (\ref{eqn:norsys}) for $a\in \{1,2,6,7\}$ :
$$
\begin{cases} 
\Xrel{1} = \X{2} \\
\Xrel{2} = \X{1}\X{3} \\
\Xrel{7} = \X{6} \\
\Xrel{6} = \X{5}\X{7}.
\end{cases}
$$
A non-zero solution $(\X{1},\X{2},\X{3},\X{5},\X{6},\X{7})$ of the above is uniquely determined by $\X{1}$ and $\X{7}$. If we set 
$$
X_0=(\X{1},\X{2},\X{3},\X{5},\X{6},\X{7})=(\ncC{1},\ncC{2},\nc{3},\nc{5},\ncC{6},\ncC{7}),
$$
then $X_0$ satisfies the above system of equation by Proposition \ref{thm:algrel}. Therefore, $\nc{a} = \ncMY{a}$ for $a\in \{1,2,3,5,6,7\}$.

When $\g$ is of type $F_4$, 
$\ncC{a}= \ncMY{a}$ for $a\in \{1,4\}$ from Propositions \ref{prop:type2} and \ref{prop:type3F4}. Writing (\ref{eqn:norsys}) for $a\in \{1,4\}$ :
$$
\begin{cases} 
\Xrel{1} = \X{2} \\
\Xrel{4} = \X{3},
\end{cases}
$$
we can conclude $\ncC{a} = \ncMY{a}$ for all $a\in \{1,2,3,4\}$.

When $\g$ is of type $G_2$, 
$\ncC{1}= \ncMY{1}$ from Proposition \ref{prop:type2}. From (\ref{eqn:norsys}) for $a=1$ we get
$$
\Xrel{1} = \X{2},
$$
which implies $\ncC{a} = \ncMY{a}$ for all $a\in \{1,2\}$.
\end{proof}

\section*{Acknowledgements}
We thank David Hernandez for his comments on an earlier version of this paper.
\bibliographystyle{amsalpha}
\bibliography{proj,proj_nonmr}
\end{document}